\documentclass[a4paper, 11pt]{amsart}
\usepackage{longtable, stackrel, multicol, fancybox, tcolorbox}
\usepackage{bm,amsfonts,amsmath,amssymb,mathrsfs,bbm,amsthm} 
\usepackage{graphicx, color,hyperref,eurosym, eucal,palatino}

\definecolor{darkblue}{rgb}{0.2,0.2,0.6}
\definecolor{darkblue2}{rgb}{0.2,0.2,0.9}
\definecolor{superdarkblue}{rgb}{0.2,0.2,0.3}
\hypersetup{colorlinks=true, linkcolor=darkblue,citecolor=darkblue, 
	urlcolor = superdarkblue}
\definecolor{citegreen}{rgb}{0.2,0.2,0.6}
\usepackage{marginnote}
\usepackage{geometry}

\newcommand\name[1]{{\small\sc#1}}

\newcommand\Op{\sfH_{\omg,\Sg}}

\newcommand\frh{\mathfrak{h}}

\newcommand\nb{\nabla}

\newcommand{\dist}{\mathop{\mathrm{dist}}\nolimits}

\theoremstyle{definition}

\usepackage[normalem]{ulem}

\makeatletter
\newcommand{\vast}{\bBigg@{3}}
\newcommand{\Vast}{\bBigg@{5}}

\usepackage{scrextend}
\changefontsizes{11pt}

\newcommand{\eg}{{\it e.g.}\,}
\newcommand{\ie}{{\it i.e.}\,}
\newcommand{\cf}{{\it cf.}\,}

\renewcommand\and{\qquad\text{and}\qquad}

\newcommand\sm{\setminus}

\newcommand\frm{\frh_{\omg,\Sg}}

\newcommand\Ev{\lm_1^\omg(\Sg)}

\newcommand\dl{\delta}
\newcommand\Dl{\Delta}

\newcommand{\comm}[1]{}

\newcommand\lm{\lambda}

\newcommand\p{\partial}

\newcommand\omg{\omega}
\newcommand\Omg{\Omega}
\renewcommand\tt{\theta}

\renewcommand{\iff}{\textit{if, and only if,}\,}

\newcommand\arr{\rightarrow}

\newcommand\Sg{\Sigma}

\newcommand\sess{\sigma_{\rm ess}}

\newcommand\dd{{\mathsf{d}}}

\newcounter{counter_a}
\newenvironment{myenum}{\begin{list}{{\rm(\roman{counter_a})}}%
{\usecounter{counter_a}
\setlength{\itemsep}{1.ex}\setlength{\topsep}{0.8ex}
\setlength{\leftmargin}{5ex}\setlength{\labelwidth}{5ex}}}{\end{list}}

\usepackage[latin1]{inputenc}
\usepackage[T1]{fontenc}

\numberwithin{figure}{section}
\numberwithin{equation}{section}
\theoremstyle{plain}
\newtheorem*{thm*}{Theorem}
\newtheorem{thm}{Theorem}

\newtheorem{prop}[thm]{Proposition}

\theoremstyle{remark}

\theoremstyle{plain}

\newcommand\ov{\overline}

\def\ov{\overline}

      \def\dC{{\mathbb C}}

      \def\dR{{\mathbb R}}

   \def\sfH{{\mathsf H}}

\def\sfV{{\mathsf V}}

      \def\cC{{\mathcal C}}
\def\cD{{\mathcal D}}      
      
      \def\cL{{\mathcal L}}

\newcommand{\dom}{\mathrm{dom}\,}

\makeatletter
\def\section{\@startsection{section}{1}\z@{.9\linespacing\@plus\linespacing}%
	{.7\linespacing} {\fontsize{14}{15}\selectfont\bfseries\centering}}

\def\subsection{\@startsection{subsection}{1}\z@{.9\linespacing\@plus\linespacing}%
	{.3\linespacing} {\fontsize{11}{12}\selectfont\bfseries}}	
	
\def\paragraph{\@startsection{paragraph}{4}%
	\z@{0.3em}{-.5em}%
	{$\bullet$ \ \normalfont\itshape}}

\@addtoreset{equation}{section}
\makeatother

\theoremstyle{definition}

\title[]{Spectral isoperimetric inequality for the \boldmath{$\dl'$}-interaction on a contour}
 
\author{Vladimir Lotoreichik}
\address{Department of Theoretical Physics, Nuclear Physics Institute, 	Czech Academy of Sciences, 25068 \v Re\v z, Czech Republic}
\email{lotoreichik@ujf.cas.cz}

\dedicatory{{\rm A contribution to the proceedings of the 
$3^{\rm rd}$ workshop}\\ Mathematical Challenges of Zero-Range Physics:
rigorous results and open problems}

\begin{document}

\begin{abstract}
	We consider the problem of geometric optimization for the lowest	eigenvalue of the two-dimensional Schr\"odinger operator with an
	attractive $\dl'$-interaction of a fixed strength, the support of
	which is a $C^2$-smooth contour. Under the constraint of a fixed length of the contour, we prove that the lowest eigenvalue is
	maximized by the circle. The proof relies on the
	min-max principle and the method of parallel coordinates.
\end{abstract}



\maketitle

\section{Introduction}

\subsection{The state of the art and motivation}
The question of optimizing shapes in spectral theory is a rich subject with many applications and deep mathematical insights;
see the monographs~\cite{Henrot, Henrot2} and the references therein.
In this note, we consider the problem of shape optimization for the lowest eigenvalue of the two-dimensional Schr\"o\-dinger operator with a $\dl'$-interaction supported on a closed contour in $\dR^2$. This problem can be regarded as a counterpart of the analysis performed in~\cite{EHL06} 
for $\dl$-interactions. 

In the recent years, the investigation of
Schr\"odinger operators with $\dl'$-interactions supported on hypersurfaces became a topic of permanent interest -- see,
\eg, \cite{BGLL15, BEL14, BLL13, EJ13, EKh15, EKh18, JL16, MPS16}. The Hamiltonians with $\dl'$-interactions and some of their 
generalizations appear, for example, in the study of photonic crystals~\cite{FK96a,FK96b} and in the analysis of the Dirac operator
with scalar shell interactions~\cite{HOP18}.
The boundary condition corresponding to the $\dl'$-interaction arises in the
asymptotic analysis of a class of structured thin Neumann obstacles~\cite{DFZ18, Kh70}.
Finally, the same boundary condition pops up in the computational spectral theory; see~\cite{Da99} and 
the references therein.

The proofs in~\cite{EHL06} and in related optimization problems for 
singular interactions on hypersurfaces~\cite{AMV16,
	BFKLR17, E05, EL17, EL18, L18}
rely on the Birman-Schwinger principle, which can also be viewed as a boundary integral reformulation of the
spectral problem.  In this note, we do not pass to
any boundary integral reformulation. Instead, we combine
the min-max principle and the method of parallel coordinates
on the level of the quadratic form for the Hamiltonian,
in the spirit of the recent analysis for the Robin Laplacian~~\cite{AFK17, FK15, KL, KL17a,KL17b}. Our main motivation is to show
that this approach initially developed for the Robin Laplacian can
also be adapted for a much wider class of
optimization problems involving surface interactions. The convenience of this alternative method is particularly visible for $\dl'$-interactions, because the operator arising in the corresponding Birman-Schwinger principle 
(\cf~\cite[Rem. 3.9]{BLL13}) is more involved than for $\dl$-interactions.

\subsection{Schr\"odinger operator with a $\dl'$-interaction on a contour}
In order to define the Hamiltonian, we need to introduce some notation. In
what follows we consider a bounded, simply connected, $C^2$-smooth domain 
$\Omg_+\subset \dR^2$, whose boundary will be denoted by $\Sg = \p\Omg_+$. The complement $\Omg_- := \dR^2\sm\ov{\Omg_+}$ of $\Omg_+$ is an unbounded exterior domain with the same boundary $\Sg$. For a function $u \in L^2(\dR^2)$ we set $u_\pm := u|_{\Omg_\pm}$. We also introduce the first 
order $L^2$-based Sobolev space on $\dR^2\sm\Sg$ as follows 
\[
	H^1(\dR^2\sm\Sg) := H^1(\Omg_+)\oplus H^1(\Omg_-),
\]
where $H^1(\Omg_\pm)$ are the conventional
first-order $L^2$-based Sobolev spaces on $\Omg_\pm$.

Given a real number~$\omg > 0$, we consider the spectral problem for the self-adjoint operator $\Op$ corresponding via the first representation theorem to the closed, densely defined, symmetric, and semi-bounded quadratic form in $L^2(\dR^2)$,
\begin{equation}\label{eq:form}
	\frm[u]
	=
	\big\|\nb_{\dR^2\sm\Sg} u\big\|^2_{L^2(\dR^2;\dC^2)}  
	- 	
	\omg\big\|[u]_\Sg\big\|^2_{L^2(\Sg)},
	\qquad
	\dom \frm  = H^1(\dR^2\sm\Sg),
\end{equation}
where $\nb_{\dR^2\sm\Sg} u := \nb u_+\oplus \nb u_-$
and $[u]_{\Sg} := u_+|_\Sg - u_-|_\Sg$ denotes the jump of the trace of $u$ on $\Sg$; \cf~\cite[Sec.~3.2]{BEL14}. The operator $\Op$ is usually called 
\emph{the Schr\"odinger operator with the $\dl'$-interaction} of strength
$\omg$ supported on $\Sg$. It acts as the minus Laplacian on functions satisfying the transmission boundary condition  of $\dl'$-type on the interface $\Sg$
\begin{equation}\label{eq:bc0}
	\p_{\nu_+} u_+|_\Sg = -\p_{\nu_-} u_-|_\Sg	= \omg[u]_\Sg,
\end{equation}
where $\p_{\nu_\pm}u_\pm|_{\Sg}$ denotes the trace
onto $\Sg$ of the normal derivative of $u_\pm$ with
the normal vector $\nu_\pm$ at the boundary of $\Omg_\pm$ pointing outwards; see Section~\ref{sec:prelim} for more details.

Recall that the essential spectrum of $\Op$ coincides
with the set $[0,\infty)$ and that its negative discrete
spectrum is known to be non-empty;
see Proposition~\ref{prop:ess_spec} below. By
$\Ev$ we denote the spectral threshold of
$\Op$, which is
an isolated negative eigenvalue.

\subsection{The main result}
The aim of this note is to demonstrate that $\Ev$
is maximized by the circle $\cC\subset\dR^2$, among all
contours of a fixed length.
A precise formulation of this statement is the content of the following theorem.
\begin{thm}\label{Thm}
	For any~$\omg > 0$, one has
	\begin{flalign*}\label{result}
		\max_{|\Sg| = L} \Ev  =	\lm_1^\omg(\cC),
	\end{flalign*}
	where $\cC\subset\dR^2$ is a circle of a given length $L > 0$ and the maximum is taken over all $C^2$-contours of length $L$.
\end{thm}
The proof of Theorem~\ref{Thm} relies on the min-max principle and the method of parallel coordinates. The
latter method has been proposed in~\cite{PW61} by
\name{L.\,E.~Payne} and \name{H.\,F.~Weinberger} in order
to obtain inequalities being reverse to the celebrated
Faber-Krahn inequality~\cite{F23,K24} with some geometrically-induced
corrections. Recently it has been observed that this method is very efficient in the proofs of isoperimetric inequalities for the lowest eigenvalue of the Robin Laplacian on bounded~\cite{AFK17,FK15} and 
exterior~\cite{KL17a, KL17b} domains with an `attractive' boundary condition. In the present paper we adapt this approach for the case of a bounded domain and its exterior coupled via the transmission boundary condition~\eqref{eq:bc0} of $\dl'$-type.
 
\subsection*{Organisation of the paper}

In Section~\ref{sec:prelim} we recall the known spectral properties of $\Op$ that are needed in this paper. Section~\ref{sec:circle} is devoted to the
spectral analysis of $\sfH_{\omg,\cC}$ with the interaction supported on a circle $\cC$. The method of parallel coordinates is briefly outlined in Section~\ref{sec:parallel}. Theorem~\ref{Thm} is
proven in Section~\ref{sec:proof}. The paper is concluded by Section~\ref{sec:discussion} containing a discussion of the obtained
results and their possible extensions.

\section{The spectral problem for the \boldmath{$\dl'$}-interaction supported on a closed contour}\label{sec:prelim}
%
Recall that we consider a bounded, simply connected, $C^2$-smooth domain 
$\Omg_+\subset \dR^2$ with the boundary $\Sg = \p\Omg_+$ and with the complement $\Omg_- := \dR^2\sm\ov{\Omg_+}$. Recall also that for a function 
$u \in L^2(\dR^2)$, we set $u_\pm := u|_{\Omg_\pm}$.
At the same time, the (attractive) coupling strength $\omg$~is a fixed positive number. 

We are interested in the spectral properties of the self-adjoint operator $\Op$ in $L^2(\dR^2)$ introduced via the first representation theorem~\cite[Thm. VI 2.1]{Kato} as associated with
the closed, densely defined, symmetric, and semi-bounded quadratic
form $\frm$ defined in~\eqref{eq:form};
see~\cite[Sec. 3.2]{BEL14} and also~\cite[Sec. 3.3 and Prop. 3.15]{BLL13}.

We would like to warn the reader that in the majority of the papers on $\dl'$-interactions not $\omg$ itself, but its inverse $\beta := \omg^{-1}$ is called the strength of the interaction. This tradition goes back to papers on point $\dl'$-interaction on the real line; see~\cite{AGHH} and the references therein. Preserving
this tradition for $\dl'$-interactions on hypersurfaces
can be physically motivated, but 
leads to a technical mathematical inconvenience, which
we would like to avoid.

Let us add a few words about the explicit characterisation of the operator $\Op$. The domain of $\Op$ consists
of functions $u \in H^1(\dR^2\sm\Sg)$, which satisfy $\Dl u_\pm \in L^2(\Omg_\pm)$ in the sense of distributions and the $\dl'$-type
boundary condition~\eqref{eq:bc0} on $\Sg$
in the sense of traces.
Moreover, for any $u \in \dom \Op$ we have
$\Op u = (-\Dl u_+) \oplus (-\Dl u_-) $. The reader may
consult~\cite[Sec. 3.2 and Thm. 3.3]{BEL14}, where
it is shown that the operator characterised above
is indeed the self-adjoint operator representing
the quadratic form $\frm$ in~\eqref{eq:form}.
It is worth mentioning that $C^2$-smoothness of $\Sg$
is not needed to define the operator $\Op$, but it is important for the method of parallel coordinates used in the proof of Theorem~\ref{Thm}.

The lowest spectral point of $\Op$ can be characterised by the min-max principle~\cite[Sec.~XIII.1]{RS4}  as follows
\begin{equation}\label{eq:minmax}
	\lm_1^\omg(\Sg) =
	\inf_{\stackrel[u\ne 0]{}{u \in H^1(\dR^2\sm\Sg)}}
	\frac{\frm[u]}{\|u\|^2_{L^2(\dR^2)}}.
\end{equation} 
It is not surprising that the operator $\Op$ has a
non-empty essential spectrum. In fact, one can show that $\Op$ is a compact perturbation in the sense of resolvent differences of the free Laplacian on $\dR^2$ and thus the essential spectrum coincides with the positive semi-axis.
Using the characteristic function of $\Omg_+$
as a test function for~\eqref{eq:minmax}  one gets that the negative
discrete spectrum of $\Op$ is non-empty. More specifically, we have the
following statement.
\begin{prop}\label{prop:ess_spec}
	For all $\omg > 0$, the following hold.
	\begin{myenum}
		\item The essential spectrum of $\Op$ is characterized as follows $\sess(\Op) =[0,\infty)$.
		\item The negative discrete spectrum
		of $\Op$ is non-empty.
	\end{myenum} 
\end{prop}
A proof of~(i) in the above proposition can be found in~\cite[Thm. 4.2\,(ii)]{BEL14}, see also~\cite[Thm. 3.14\,(i)]{BLL13}. A proof of~(ii) is contained in~\cite[Thm. 4.4]{BEL14}. Some further properties of the discrete spectrum of $\Op$ are investigated in or follow from~\cite{BLL13, EJ13}.
Note that by~\cite[Thm. 3.14\,(ii)]{BLL13} the negative discrete spectrum of $\Op$ is finite for $C^\infty$-smooth
$\Sg$ and it can be shown in a similar way that 
the discrete spectrum persists to be finite for $C^2$-smooth $\Sg$.

Taking that the spectral threshold of $\Op$ is a negative discrete eigenvalue into account, we can
slightly modify the characterisation of $\Ev$ given in~\eqref{eq:minmax}
as follows:
\begin{equation}\label{eq:minmax2}
	\Ev =
	\inf_{\stackrel[{\frm[u]<0}]{}{u \in H^1(\dR^2\sm\Sg)}}
	\frac{\frm[u]}{\|u\|^2_{L^2(\dR^2)}}.
\end{equation} 

\section{The spectral problem for the \boldmath{$\dl'$}-interaction supported on a circle}\label{sec:circle}
In this section we consider the lowest eigenvalue
for the operator $\sfH_{\omg,\cC}$ with the $\dl'$-interaction of strength $\omg > 0$ supported on a circle $\cC = \cC_R\subset\dR^2$ of radius $R > 0$. Our primary
interest concerns the dependence of this eigenvalue
on the radius $R$.
For the sake of convenience,
we introduce the polar coordinates $(r,\tt)$, whose pole 
coincides with the center
of $\cC$. Note also that the circle $\cC$
splits the Euclidean plane $\dR^2$ into the disk 
$\cD_+ = \{x\in\dR^2\colon |x| < R\}$ and its exterior
$\cD_- = \{x\in\dR^2\colon |x| > R\}$.
\begin{prop}\label{prop:circle}
	Let $\cC = \cC_R\subset\dR^2$ be a circle of radius $R > 0$.
	Let $\lm_1^\omg(\cC) = -k^2 < 0$ and $u_1\in H^1(\dR^2\sm\cC)$ be, respectively, 
	the lowest eigenvalue and a corresponding eigenfunction of 
	$\sfH_{\omg,\cC}$. 
	Then the following hold. 
	\begin{myenum}
		\item The value $k > 0$ is the unique positive solution of the equation 
		\begin{equation*}\label{eq:secular}
			k^2RI_1(kR)K_1(kR) = \omg.   
		\end{equation*}
		\item
		The function $(0,\infty)\ni R\mapsto\lm_1^\omg(\cC_R)$ is continuous, increasing, and 
		\begin{equation*}\label{eq:limits}
			\lim_{R\arr 0^+} \lm_1^\omg(\cC_R) = -\infty
			\and
			\lim_{R\arr \infty} \lm_1^\omg(\cC_R) = -4\omg^2.
		\end{equation*}	
		\item
		The ground-state $u_1$ is radial and can be expressed in polar coordinates $(r,\tt)$ as
		\begin{equation}\label{eq:eigenfunction}
			u_1(r,\tt) = 
			\begin{cases}
			K_1(kR) I_0(kr),&\quad r < R,\\
			-I_1(kR) K_0(kr),&\quad r > R.
			\end{cases}
		\end{equation}
	\end{myenum}	
\end{prop}
\begin{proof}
	In view of the radial symmetry
	of the problem, the eigenfunction $u_1\in L^2(\dR^2)$
	must necessarily be radially symmetric as well. Therefore,
	in polar coordinates $(r,\tt)$ we have $u_1(r,\tt) = \psi(r)$. Using this simple observation we see that $\lm_1^\omg(\cC) = -k^2 < 0$ \iff the following ordinary differential spectral problem
	\begin{equation}\label{ODE}
	\Vast \{ 
	\begin{aligned}
	-r^{-1}\big[r \psi'(r)]' & =-k^2\psi(r)
	&&\quad\mbox{for}\quad  r\in \dR_+\sm \{R\},\\[0.3ex]
	\psi'(R^-)  = \psi'(R^+) & = \omg\big[\psi(R^-) - \psi(R^+)\big],
	\\[0.3ex]
	\lim_{r\arr \infty} \psi(r)  = \psi'(0) & = 0,
	\end{aligned}
	\end{equation}
	possesses a solution $(\psi, k)$ with
	$\psi\neq 0$, $k > 0$; \cf~\cite[Sec.~2]{AFK17}
	and~\cite[Sec. 3]{KL17a}.
	Observe that the general solution of the differential equation in~\eqref{ODE}
	with $k > 0$ is given by
	\[
	\psi(r) = 
	\begin{cases}
	A_+K_0(kr) +  B_+ I_0(kr),&\quad r < R,\\[0.3ex]
	A_-K_0(kr) +  B_- I_0(kr),&\quad r > R,
	\end{cases}
	\]
	where $A_\pm, B_\pm \in\dC$ are some coefficients and $K_0(\cdot)$, $I_0(\cdot)$ are the modified Bessel functions of zero order.
	Taking into account the boundary conditions at infinity and at the origin from~\eqref{ODE}
	and using the behaviour of $K_0(x)$ and $I_0(x)$ and of their derivatives 
	for large~\cite[9.7.1-4]{AS64}
	and small~\cite[9.6.7-9]{AS64}
	values of $x$
	we conclude that $A_+ = B_- = 0$.
	Thus, the expression for $\psi$ simplifies to
	\[
	\psi(r) = 
	\begin{cases}
	B_+ I_0(kr),&\quad r < R,\\[0.3ex]
	A_-K_0(kr) ,&\quad r > R,
	\end{cases}
	\]
	where the constants~$B_+$ and $A_-$ must not both be zero 
	to get a non-trivial solution.
	Differentiating~$\psi$ with respect to $r$, 
	we find
	\[
		\psi'(r) = 
		\begin{cases}
		kB_+ I_1(kr),&\quad r < R,\\[0.3ex]
		-kA_- K_1(kr) ,&\quad r > R.
		\end{cases}
	\]
	Thus, the boundary condition in~\eqref{ODE} 
	at the point $r = R$
	yields the requirement
	\[
	\begin{cases}
		B_+ I_1(kR) = -A_- K_1(kR),\\
		\omg\big( B_+I_0(kR) -A_-K_0(kR)\big) =B_+k I_1(kR).\\
	\end{cases}
	\]
	This linear system of equations can be simplified as 
	\begin{equation}\label{eq:system}
	\begin{cases}
		A_- K_1(kR) + B_+ I_1(kR) = 0,\\
		-A_- \omg K_0(kR) + B_+\big(\omg I_0(kR) - kI_1(kR)\big) = 0.\\
	\end{cases}
	\end{equation}
	The existence of a non-trivial solution for the system
	above is equivalent to vanishing of the underlying determinant, which gives us a scalar equation
	on $k$
	\begin{equation}\label{eq:scalar}
		\omg\left [
		K_1(kR)I_0(kR) + I_1(kR)K_0(kR)
		\right ] - kK_1(kR)I_1(kR) = 0.
	\end{equation}
	Provided $k > 0$ is a solution of~\eqref{eq:scalar},
	the vector $(A_-,B_+)^\top = (I_1(kR), -K_1(kR))^\top$
	is a solution of the system~\eqref{eq:system} and, hence, the expression~\eqref{eq:eigenfunction} for the ground-state $u_1$ immediately follows.
	
	Furthermore, using the identity $K_1(x)I_0(x) + I_1(x)K_0(x) = x^{-1}$ (see~\cite[9.6.15]{AS64}) we simplify~\eqref{eq:scalar} as
	\begin{equation}\label{eq:implicit}
		k^2R K_1(kR)I_1(kR) = \omg. 
	\end{equation}
	Consider now the $C^\infty$-smooth function
	$F(x) := x K_1(x)I_1(x)$ on $(0,\infty)$ in more detail.
	The analysis in~\cite[Prop 7.2]{HW74} implies that $F'(x) > 0$ and $\lim_{x\arr 0^+} F(x) = 0$,
	$\lim_{x\arr\infty}F(x) = \frac12$. 
	Hence, the function $G(k) := kF(kR)$ in the left-hand side of~\eqref{eq:implicit} is strictly
	increasing in $k$ and satisfies
	$\lim_{k\arr 0^+} G(k) = 0$, $\lim_{k\arr\infty}G(k) = \infty$, $G(k) < \frac{k}{2}$.  
	Therefore, the equation~\eqref{eq:implicit} possesses a
	unique positive solution $k_\star = k_\star(R) > 0$ 
	satisfying the bounds
	\begin{equation}\label{eq:k_bounds}
		2\omg < k_\star(R) < \frac{\omg}{F(2\omg R)}.
	\end{equation}
	Consequently, we get $\lim_{R\arr \infty} k_\star(R) = 2\omg$ and, hence, $\lim_{R\arr\infty}\lm_1^\omg(\cC_R) = -4\omg^2$.
	
	Using the implicit function theorem~\cite[Thm. 3.3.1]{KP} we find that $(0,\infty)\ni R\mapsto k_\star(R)$ is a $C^\infty$-smooth function,
	whose derivative satisfies
	\[
		k_\star'(R) F(k_\star(R)R) 
		+ 
		\big(k_\star(R) + Rk_\star'(R)\big)	F'(k_\star(R)R) = 0. 
	\]
	The above equation yields
	\[
		k_\star'(R) 
		= 
		-\frac{k_\star(R)F'(k_\star(R)R)}{F(k_\star(R)R) + RF'(k_\star(R)R)} < 0.
	\]
	Hence, the function $R\mapsto \lm_1^\omg(\cC_R)= -\big(k_\star(R)\big)^2$ 	is increasing.

	Finally, using the characteristic function $\chi_{\cD_+}\in H^1(\dR^2\sm\cC)$
	of the disk $\cD_+$ as a test function
	we get
	\[
		\lm_1^\omg(\cC_R) 
		\le 
		\frac{\frh_{\omg,\cC}[\chi_{\cD_+}]}
			{\|\chi_{\cD_+}\|^2_{L^2(\dR^2)}}
		=
		\frac{-2\pi R\omg}{\pi R^2}
		= 
		-\frac{2\omg}{R} \arr -\infty,
		\qquad \text{as}\,\,		R\arr 0^+. \qedhere
	\]
\end{proof}	

\section{The method of parallel coordinates}
\label{sec:parallel}
In this section we briefly recall the method of parallel coordinates. We follow the modern presentation in~\cite{S01} with an adjustment
of notation. Further details and proofs can be found in the classical papers~\cite{F41, H64}, see also the monograph~\cite{Ba80} and the references therein.

First, we introduce the distance-functions on the domains $\Omg_\pm$
as
\begin{equation*}\label{key}
	\rho_\pm\colon\Omg_\pm\arr\dR_+,\qquad
	\rho_\pm(x) := \dist(x, \Sg).
\end{equation*}
The functions $\rho_\pm$
are Lipschitz continuous with the Lipschitz constant $=1$
\begin{equation}\label{eq:rho_Lip}
	|\rho_\pm(x) - \rho_\pm(y)| \le |x-y|,\qquad \forall\, x,y \in \Omg_\pm.
\end{equation}
For the convenience of the reader we will show~\eqref{eq:rho_Lip} for $\Omg_-$. Without loss of generality we suppose that $\rho_-(x)
\ge \rho_-(y)$. Let $z\in \Sg$ be such that 
$\rho_-(y) = |y-z|$. Hence, we obtain that
$\rho_-(x) \le |x-z|$. Thus, we get
\[
	|\rho_-(x) - \rho_-(y)| \le |x-z| - |y-z|
	\le |x-y|,
\]
where the last step follows from the
triangle inequality in $\dR^2$.

Furthermore, we introduce the in-radii of $\Omg_\pm$ by
\begin{equation*}\label{eq:in_radius}
	R_\pm := \sup_{x\in \Omg_\pm} \rho_\pm(x).
\end{equation*}
The in-radius of $\Omg_+$ is thus the radius of the largest disk in $\dR^2$ that can be inscribed into $\Omg_+$, and due to the standard well-known isoperimetric inequality $|\Sg|^2 \ge 4\pi|\Omg_+|$ we get 
\begin{align}\label{isop2}
	R_+ \le R,\qquad\text{where}\quad 
	R = \frac{L}{2\pi}. 
\end{align}
On the other hand, we obviously  have
$R_- = \infty$.

Finally, we introduce the following auxiliary functions
\begin{equation}\label{eq:LmAm}
\begin{aligned}
	L_\pm &\colon [0, R_\pm] \to \dR_+,& 
	\qquad 
	L_\pm(t)&:=
	\big|	\{x\in\ov{\Omg_\pm}\colon \rho_\pm(x)= t\}\big|, \\
	A_\pm &\colon [0, R_\pm]\to [0, | \Omg_\pm| ],& 
	\qquad
	 A_\pm(t)&:=
	\big| \{x \in\Omg_\pm\colon \rho_\pm(x) < t\}\big|.
\end{aligned}
\end{equation}
Clearly, $L_\pm(0)= L$ and $A_+(R_+)= 
|\Omg_\pm|$. 
The value $A_\pm(t)$ is simply the area of the sub-domain of $\Omg_\pm$, which consists of the points located at the distance less that $t$ from its boundary $\Sg$. On the other hand, $L_\pm(t)$ is the length of the corresponding level
set of the function $\rho_\pm$.

Some analytic properties of the functions in~\eqref{eq:LmAm} are summarized in the following proposition.
\begin{prop}{\cite[App. 1, Prop.~A.1]{S01}, \cite[Chap. I, Sec. 3.6]{Ba80}}\label{prop_savo}
	Let the functions $A_\pm$ and $L_\pm$ be as in~\eqref{eq:LmAm}. Then the following hold.
	\begin{myenum}
		\item $A_\pm$ is continuous,
		locally Lipschitz, and increasing.
		%
		\item $A_\pm'(t)= L_\pm(t) > 0$
		for almost every $t\in [0,R_\pm]$.
		\item $L_+(t) \leq L - 2\pi t$
		and $L_-(t) \leq L + 2\pi t$.
	\end{myenum}
\end{prop}
Further, let $\psi_+ \in C^\infty([0,R])$
and $\psi_- \in C^\infty_0([0,\infty))$ be 
arbitrary and real-valued. Due to the properties of $A_\pm$ stated in Proposition~\ref{prop_savo}\,(i), there exist Lipschitz continuous 
functions $\phi_+\colon [0, |\Omg_+|]\arr \dR$
and $\phi_-\colon [0, \infty)\arr \dR$ satisfying
\begin{align}\label{parallel1}
	\psi_+\big|_{[0, R_+]} = \phi_+ \circ A_+
	\and
	\psi_- = \phi_- \circ A_-.
\end{align}
Consider now the test function 
\[
	u = 
	(\phi_+ \circ A_+ \circ \rho_+)\oplus
	(\phi_- \circ A_- \circ \rho_-).
\]
Lipschitz continuity of $\phi_\pm$,
Proposition~\ref{prop_savo}\,(i) and~\eqref{eq:rho_Lip}
imply that  $u \in H^1(\dR^2\sm\Sg)$. Employing the parallel coordinates together with the co-area formula (see~\cite[Eq. 30]{S01}
for more details) and applying further~\eqref{isop2},~\eqref{parallel1} we get
\begin{equation}\label{parallel2}
\begin{aligned}
	&\|\nb_{\dR^2\sm\Sg} u \|^2_{L^2(\dR^2;\dC^2)}
	= 
	\|\nb u_+ \|^2_{L^2(\Omg_+;\dC^2)} 
	+ 
	\|\nb u_- \|^2_{L^2(\Omg_-;\dC^2)}\\
	&\qquad 
	=
	\int_0^{R_+} | \phi_+'(A_+(t))|^2 
	(A_+'(t))^3 \dd t
	+
	\int_0^\infty | \phi_-'(A_-(t))|^2 
	(A_-'(t))^3 \dd t
	\\
	&\qquad =
	\int_0^{R_+} | \psi_+'(t)|^2 
	A_+'(t) \dd t
	+
	\int_0^\infty | \psi_-'(t)|^2 
	A_-'(t) \dd t\\
	&\qquad \le
	\int_0^{R} | \psi_+'(t)|^2 
	(L-2\pi t) \dd t
	+
	\int_0^\infty | \psi_-'(t)|^2 
	(L+2\pi t) \dd t,
\end{aligned}
\end{equation}
where Proposition~\ref{prop_savo}\,(ii),~(iii) was used in the last step.
Following the same steps (\cf~\cite[App. 1]{S01}) we also get
\begin{equation}\label{parallel3}
\begin{aligned}
	\| u\|^2_{L^2(\dR^2)} 
	& =
	\int_0^{R_+} | \phi_+(A_+(t))|^2 
	A_+'(t) \dd t
	+
	\int_0^\infty | \phi_-(A_-(t))|^2 
	A_-'(t) \dd t
	\\
	& \le
	\int_0^{R} | \psi_+(t)|^2 
	(L-2\pi t)\dd t
	+
	\int_0^\infty | \psi_-(t)|^2 
	(L + 2\pi t) \dd t.
\end{aligned}
\end{equation}
Let us focus on the jump of the trace of $u$
onto $\Sg$. It is easy to see that for any 
$x\in\Sg$ we have $[u]_\Sg(x) = \psi_+(0) - \psi_-(0)$. Hence, we obtain
\begin{align}\label{parallel4}
	\big\| [u]_\Sg\big\|^2_{L^2(\Sg)} =
	L |\psi_+(0) - \psi_-(0)|^2. 
\end{align}

\section{Proof of Theorem~\ref{Thm}}
\label{sec:proof}

We are now able to conclude the proof of Theorem~\ref{Thm}. The argument will be split into two steps.

\smallskip

\noindent {\em Step 1.}
On this step, we make several preliminary constructions.
First, we define the sub-space of $H^1(\dR^2\sm\cC)$ as
\begin{equation*}\label{eq:subspace}
	\cL := \big\{
	w = w_+\oplus w_- \in C^\infty(\ov{\cD_+})\oplus C^\infty_0(\ov{\cD_-})\colon \p_\tt w = 0\big \}.  
\end{equation*}
Notice that for any $w \in\cL$ there exist functions $\psi_+\in C^\infty([0,R])$ and $\psi_-\in C^\infty_0([0,\infty))$ satisfying 
$w_+(r,\tt) =  \psi_+(R - r)$ and
$w_-(r,\tt) =  \psi_-(r - R)$. 
Next, we point out that
the ground-state $u_1\in H^1(\dR^2\sm\cC)$ of $\sfH_{\omg,\cC}$
given in~\eqref{eq:eigenfunction} belongs
to the closure of $\cL$ in the norm of $H^1(\dR^2\sm\cC)$; \ie there exists a sequence
$(w_n)_n \in \cL$ such that
\begin{equation}\label{eq:density}
	\|w_n - u_1\|_{H^1(\dR^2\sm\cC)}
	\arr 0,\qquad n\arr\infty.
\end{equation}
Finally, we define the linear mapping 
$\sfV \colon \cL\arr H^1(\dR^2\sm\Sg)$ by
\begin{equation*}\label{key}
	(\sfV w)(x) := 
	\begin{cases} 
		\psi_+(\rho_+(x)),&\quad x\in\Omg_+,\\
		\psi_-(\rho_-(x)),&\quad x\in\Omg_-.\\
	\end{cases}	
\end{equation*} 

\noindent {\em Step 2.}
Using the inequalities~\eqref{parallel2},~\eqref{parallel3}
and the identity~\eqref{parallel4}, we obtain from the min-max principle~\eqref{eq:minmax2} that
\begin{equation*}\label{cl1}
\begin{aligned}
	\Ev 
	& \le 
	\inf_{w\in \cL\colon \frh_{\omg,\cC}[w] <0} \frac{\frm[\sfV w]}{\|\sfV w\|^2_{L^2(\dR^2)}}\\
	& \le
	\inf_{w\in \cL\colon \frh_{\omg,\cC}[w] <0} \frac{\frh_{\omg,\cC}[w]}{\|w\|^2_{L^2(\dR^2)}}	
	=
	\frac{\frh_{\omg,\cC}[u_1]}{\|u_1\|^2_{L^2(\dR^2)}}
	=
	\lm_1^\omg(\cC),
\end{aligned}
\end{equation*}
where the property~\eqref{eq:density} was used
in the last but one step.

\section{Discussion}\label{sec:discussion}

The same technique can be used to reprove the optimization result in~\cite{EHL06} 
on $\dl$-interactions without making use of the Birman-Schwinger principle. In fact, the method 
seems to be applicable for a
larger sub-class of general four-parametric boundary conditions,
considered in~\cite{ER16}. One has only to ensure that the
lowest spectral point is indeed a negative eigenvalue and that the corresponding ground-state is real-valued and radially symmetric for the case of the interaction supported on a circle.

For the moment, it is unclear how to prove a counterpart of Theorem~\ref{Thm} and whether it is true or not under the constraint of a fixed area. In contrast to the case of the Robin Laplacian on an exterior domain~\cite{KL17a, KL17b}, this result
does not follow from the corresponding inequality
under the constraint of a fixed perimeter, because the lowest
eigenvalue for the $\dl'$-interaction supported
on a circle is not a decreasing, but an increasing function of its radius; see Proposition~\ref{prop:circle}. The same 
problem arises for the Robin Laplacian on a bounded
domain with a negative boundary parameter~\cite{AFK17,FK15}.

\section*{Acknowledgement}
The author is indebted to Pavel Exner, Michal Jex, 
David Krej\v{c}i\v{r}\'{i}k, and Magda Khalile for fruitful discussions and gratefully acknowledges the support by the grant No.~17-01706S of the Czech Science Foundation (GA\v{C}R). 

The author also thanks Alessandro Michelangeli for the invitation to participate in and give a mini-course at the $3^{\rm rd}$ workshop: 
{\it Mathematical Challenges of Zero-Range Physics:
rigorous results and open problems} and 
Istituto Nazionale di Alta Matematica ``Francesco Severi'' for the financial support of the travel.

\newcommand{\etalchar}[1]{$^{#1}$}

\end{document}